\documentclass[11pt,letter]{amsart}

\usepackage[utf8]{inputenx}

\usepackage{appendix}
\usepackage{textcomp}
\usepackage[english]{babel}
\usepackage[T1]{fontenc}
\usepackage{bm}
\usepackage{amssymb,amsthm,amsmath}
\usepackage{indentfirst}
\usepackage{color}
\usepackage{graphicx}
\usepackage{enumerate}
\usepackage{enumitem}
\usepackage{caption}
\usepackage{empheq}
\usepackage{xcolor}
\usepackage{color}
\usepackage{amssymb}
\usepackage{booktabs}
\usepackage{graphicx}
\usepackage[font=small,labelfont=bf,labelsep=period,tableposition=top]{caption}
\usepackage{hyperref}

\usepackage{color}

\newtheorem{teo}{Theorem}[section]
\newtheorem{prop}[teo]{Proposition}
\newtheorem{lm}[teo]{Lemma}
\newtheorem{coro}[teo]{Corollary}
\newtheorem{rem}[teo]{Remark}

\newcommand{\RR}{{\mathbb{R}}}

\newcommand{\inspol}{{\mathcal{A}}}

\newcommand{\der}{\partial}
\newcommand{\ep}{\varepsilon}

\newcommand{\om}{\Omega}
\newcommand{\dive}{\text{\normalfont div}}

\def\qed{\hfill$\square$\vspace{0.5cm}}    

\numberwithin{equation}{section}

\begin{document}

\title[]{Lipschitz stable determination of polygonal conductivity inclusions in a layered medium from the Dirichlet to Neumann map }

\author[E.~Beretta et al.]{Elena~Beretta}
\address{Dipartimento di Matematica ``Brioschi'',
Politecnico di Milano and New York University Abu Dhabi}
\email{elena.beretta@polimi.it}
\author[]{Elisa~Francini}
\address{Dipartimento di Matematica e Informatica ``U. Dini'',
Universit\`{a} di Firenze}
\email{elisa.francini@unifi.it}
\author[]{Sergio~Vessella}
\address{Dipartimento di Matematica e Informatica ``U. Dini'',
Universit\`{a} di Firenze}
\email{sergio.vessella@unifi.it}


\begin{abstract} Using a distributed representation formula of the Gateaux derivative of the Dirichlet to Neumann map with respect to movements of a polygonal conductivity inclusion, \cite{BMPS}, we extend the results obtained in \cite{BF} proving global Lipschitz stability for the determination of a polygonal conductivity inclusion embedded in a layered medium from knowledge of the Dirichlet to Neumann map.

\end{abstract}

\maketitle

%
%
\section{Introduction}
In this paper we consider the inverse problem of determining a polygonal conductivity inclusion in a layered medium. We address the issue of stable reconstruction from knowledge of the Dirichlet-to Neumann map proving a quantitative Lipschitz stability estimate. 
This extends the results obtained in \cite{BF} where Lipschitz stability was proved in the case of one or more well separated polygonal inclusions embedded in a homogeneous medium. There, a crucial step to prove Lipschitz stability was an accurate investigation of the differentiability properties of the Dirichlet to Neumann map and a lower bound of the directional Gateaux  derivative of the Dirichlet to Neumann map with respect to movements of a polygonal conductivity inclusion. In \cite{BF} the authors used a boundary representation of the directional derivative obtained rigorously in \cite{BFV17} via an accurate study of the blow-up rate of the gradient of solutions in a neighborhood of the vertices of the polygon. 
In the case of an arbitrary polygonal regular partition of the domain $\om$ the behaviour of the solutions in a neighborhood of the points of intersection of the sides of the polygons depend on how the sides of elements of the partition intersect at those points not allowing in general the derivation of the boundary representation of the derivative, \cite{BFV}.
On the other hand,  as shown in \cite{BMPS},\cite{L},\cite{LS}, one can prove the existence of  a domain representation of the Gateaux derivative, also known as distributed derivative, which is more general than the boundary representation as it is well-defined for very general partitions.\\
The novelty of our approach is to use the domain representation and to derive continuity properties and a lower bound of the derivative which finally lead to quantitative Lipschitz stability estimates in terms of the Dirichlet-to-Neumann map.\\
 To our knowledge, stability for polygonal inclusions from finitely many boundary measurements has been derived only in the case of polygonal inclusions embedded in an homogeneous medium: in  \cite{BFI} where the authors derive a local stability result and in \cite{LT} where  the  logarithmic stability estimate for convex polygons is global but the measurement depends on the unknown polygon.   Also, the results obtained in \cite{AS} where Lipschitz stability has been proved for conductivities belonging to a finite dimensional subspace do not apply in our case. On the other hand, in several applications, like the geophysical one, many measurements are at disposal justifying the use of the full Dirichlet-to-Neumann map, \cite{BCFLM}.
One of the main tools in the proof of the lower bound for the derivative is quantitative estimation of the propagation of smallness for solutions to the conductivity equation. In order to obtain sharper estimates in the lower bounds  we use some recent results obtained in \cite{CW} where a three-spheres inequality in layered media has been derived.
The result can be extended also to a more general piecewise linear partition of the domain $\om$ unless the distance of the vertices of the polygon is at controlled positive distance from the layers of the partition. 
\begin{figure}
	\centering
	\includegraphics[scale=0.7]{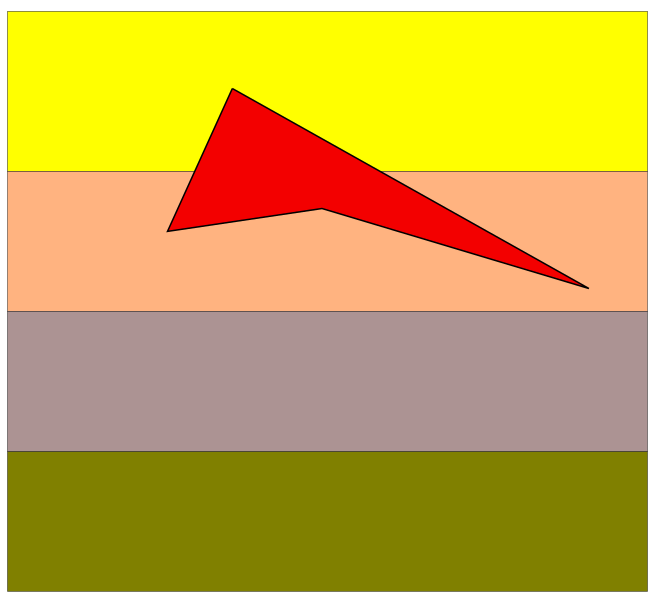}
	\caption{An admissible geometry.}
	\label{fig:geom_set}
\end{figure}
This kind of geometrical setting originates from applications, for example in geophysics exploration, where a layered medium (the earth) under inspection contains heterogeneities in the form of rough bounded subregions (for example subsurface salt bodies) with different conductivity properties \cite{LAD},\cite{F}.\\
Concerning the more realistic three dimensional setting, in \cite{ABFV} we are able to extend the ideas introduced here and prove Lipschitz stability in the case of an arbitrary polyhedral Lipschitz inclusions imbedded in a homogeneous medium and we expect that the result holds also in the case of polyhedra embedded in a layered medium.\\
The paper is organized as follows:  Section 1 contains the Introduction, in Section 2 we state our main assumptions and the main result. In Section 3 we derive a first rough stability estimate needed to prove the Lipschitz stability.  In Section 4 we state and prove some crucial regularity results and in Section 5 we prove the lower bound for the derivative and finally we establish our main stability result.



\section{Assumptions and main result}
%
%
Let $\om=(-L,L)\times(-L,L)\subset\RR^2$ and let $\omega_0,\dots,\omega_m$ be real numbers such that  $\omega_0=-L<\omega_1<\cdots<\omega_m=L$ and let
\[\overline{\om}=\cup_{i=1}^{m}\overline{\om}_i,\,\,\,\textrm{ where }
 \om_{i}=\{(x,y)\in \om: \omega_{i-1}<y<\omega_{i}\}\]

Let us consider a background conductivity of the form
\begin{equation}
	\label{1}
	\gamma_b(x)=\sum_{i=1}^m\gamma_i\chi_{\Omega_i}
\end{equation}
for positive constants $\gamma_i,i=1,\dots,m$.
Let us denote by $\Sigma_i$ the interface $(-L,L)\times\{\omega_i\}$. We assume
\begin{equation}\label{distanzastrati}
dist(\Sigma_i,\Sigma_{i+1})\geq d_0, \forall i=0,\dots,m-1.
\end{equation}
Let $k$ be a positive constant $k\neq \gamma_i,\,i=1,\dots,m$, let $\mathcal{P}\subset \om$ be a closed polygon and set
\[\gamma_{\mathcal{P}}=\gamma_b(x)+(k-\gamma_b(x))\chi_{\mathcal{P}}(x).\]
We denote by $dist(\cdot,\cdot)$ the euclidean distance between points or subsets in $\RR^2$. 

Let us denote by $\Lambda_{\mathcal{P}}$ the Dirichlet to Neumann map related to $\gamma_{\mathcal{P}}$, that is the map
\[\begin{array}{rcl}\Lambda_{\mathcal{P}}: H^{1/2}(\der\om)&\to& H^{-1/2}(\der\om)\\
f&\to&\gamma_{\mathcal{P}}{\frac{\der u}{\der n}}_{|_{\der\om}},\end{array}\]
where $u\in H^1(\om)$ is the unique solution to
\begin{equation}\label{2}
    \left\{\begin{array}{rcl}
             \dive(\gamma_{\mathcal{P}}\nabla u) & = & 0\mbox{ in }\om, \\
             u&=&f\mbox{ on }\der\om,\\
           \end{array}
    \right.
\end{equation}
and $n$ is the unit outer normal direction to $\der\om$.
The norm of the DN map in the space of linear operators $\mathcal{L}(H^{1/2}(\der\om), H^{-1/2}(\der\om))$ is defined by
\[\|\Lambda_{\mathcal{P}}\|_*=\sup\left\{\|\Lambda_{\mathcal{P}}\phi\|_{H^{-1/2}(\der\om)}/\|\phi\|_{H^{1/2}(\der\om)}\,:\,\phi\neq 0\right\}. \]

Let us know define the class of polygons we are dealing with.

Let $\inspol$ the set of closed, simply connected, simple polygons $\mathcal{P}\subset \om$ such that:
\begin{equation}\label{lati}
\mathcal{P}\mbox{ has at most }N_0\mbox{ sides each one with length greater than }d_0;
\end{equation}
\begin{equation}\label{lip}\der\mathcal{P}\mbox{ is of Lipschitz class with constants }r_0\mbox{ and }K_0,\end{equation}
there exists a  constant $\beta_0\in (0,\pi/2]$ such that the angle $\beta$ in each vertex of $\mathcal{P}$ satisfies the conditions
\begin{equation}\label{angoli}
\beta_0\leq\beta\leq 2\pi-\beta_0\mbox{ and } |\beta-\pi|\geq\beta_0,
\end{equation}
\begin{equation}\label{distanzafrontiera}
dist(\mathcal{P},\der\om)\geq d_0.
\end{equation}
We also assume that for every vertex $P_k$ of $\mathcal{P}$ we have
\begin{equation}\label{distanzainterfaccia}
dist(P_k,\Sigma_i)\geq d_0/2, \forall i=1,\dots,m-1.
\end{equation}

Notice that we do not assume convexity of the polygon.

We also assume that we can distinguish the inclusion from the background, so we assume there is a positive constant $c_0$ such that
\begin{equation}\label{contrasto}
|k-\gamma_i|\geq c_0,\text{ for }i=1,\dots, m.
\end{equation}

The constants $N_0$, $d_0$, $r_0$, $K_0$, $L$, $\beta_0$,  $c_0$, and $m$ will be referred to as the \textit{a priori data}.

In the sequel we will introduce a number of constants that we will usually denote by $C$.
The values of these constants might differ from one line to the other but they will only be determined by the a priori data and will always be greater than $1$.

Finally, let us recall the definition of the Hausdorff distance between two sets $A$ and $B$:
\[d_H(A,B)=\max\{\sup_{x\in A}\inf_{y\in B}dist(x,y),\sup_{y\in B}\inf_{x\in A}dist(y,x)\}.\]

The following stability results holds:

\begin{teo}\label{mainteo}
Let $\mathcal{P}^0,\mathcal{P}^1\in\inspol$ and let $k,\gamma_i,\,i=1,\dots,m$ satisfy assumption \eqref{contrasto}.
There exists  $C$ depending only on the a priori data such that
\[d_{H}\left(\der \mathcal{P}^0,\der\mathcal{P}^1\right)\leq C
\|\Lambda_{\mathcal{P}^0}-\Lambda_{\mathcal{P}^1}\|_*.\]
\end{teo}

\begin{teo}\label{stabint}
Let $\mathcal{P}^0,\mathcal{P}^1\in\inspol$. There exists a positive constant  $C$ depending only on the a priori data such that,
\begin{equation}\label{stimaL1}
\|\gamma_{\mathcal{P}^0}-\gamma_{\mathcal{P}^1}\|_{L^1(\Omega)}\leq C\|\Lambda_{\mathcal{P}^0}-\Lambda_{\mathcal{P}^1}\|_*.
\end{equation}
\end{teo}

%
\section{Logarithmic stability estimates}
As in \cite{BF} we notice that, thanks to Lemma 2.2 in \cite{MP} we can apply Theorem 1.1 in \cite{CFR} and have the following
\begin{prop} (3.1 in \cite{BF})
There exists $\alpha<1/4$ and $C>1$ depending only on the a priori data, such that
\begin{equation}\label{log}
\|\gamma_{\mathcal{P}^0}-\gamma_{\mathcal{P}^1}\|_{L^2(\om)}\leq C\left|\ln\|\Lambda_{\mathcal{P}^0}-\Lambda_{\mathcal{P}^1}\|_*\right|^{-\alpha}
\end{equation}
if $\|\Lambda_{\mathcal{P}^0}-\Lambda_{\mathcal{P}^1}\|_*<1/2$.
\end{prop}
\begin{coro}\label{coro1} Under the assumptions of Theorem \ref{mainteo}, there exists $\alpha<1/4$ and $C>1$ depending only on the a priori data, such that, if $\|\Lambda_{\mathcal{P}^0}-\Lambda_{\mathcal{P}^1}\|_*<1/2$, then
\begin{equation}\label{4}
	|\mathcal{P}^1\Delta \mathcal{P}^0|\leq  \frac{C}{c_0}\left|\ln\|\Lambda_{\mathcal{P}^0}-\Lambda_{\mathcal{P}^1}\|_*\right|^{-2\alpha}.
\end{equation}
\end{coro}
\begin{proof}
It follows immediately from \eqref{log} and from assumption \eqref{contrasto}, that gives
\begin{eqnarray*}
\|\gamma_{\mathcal{P}^0}-\gamma_{\mathcal{P}^1}\|^2_{L^2(\om)}&=&
\int_\om\left(\gamma_{\mathcal{P}^0}-\gamma_{\mathcal{P}^1}\right)^2=\nonumber\\
&=&\int_{\mathcal{P}^1\setminus \mathcal{P}^0}(k-\gamma_b(x))^2+\int_{\mathcal{P}^0\setminus \mathcal{P}^1}(k-\gamma_b(x))^2\nonumber\\
&\geq&c_0^2|\mathcal{P}^1\Delta \mathcal{P}^0|.
\end{eqnarray*}
\end{proof}
We can now follow the same strategy as in \cite{BF} and notice that a priori assumptions on the set of polygons $\inspol$ gives a relation between $|\mathcal{P}^1\Delta \mathcal{P}^0|$, $d_H(\der\mathcal{P}^0,\der\mathcal{P}^1)$ and
the distance between endpoint of the polygons.

Let us recall here some results from \cite{BF}.
\begin{lm}\label{lemma3.2}(Lemma 3.2 in \cite{BF})
Given two polygons $\mathcal{P}^0$ and $\mathcal{P}^1$ in $\inspol$, we have
\[d_H(\der\mathcal{P}^0,\der\mathcal{P}^1)\leq C\sqrt{\left|\mathcal{P}^0\Delta\mathcal{P}^1\right|}\]
where $C$ depends only on the a priori data.
\end{lm}
\begin{prop}\label{prop3.3}(Proposition 3.3 in \cite{BF})
Given the set of polygons $\inspol$ there exist $\delta_0$ and $C_0$ depending only on the a priori data  such that, if for some $\mathcal{P}^0$, $\mathcal{P}^1\in \inspol$ we have
\[d_H(\der\mathcal{P}^0,\der\mathcal{P}^1)\leq \delta_0,\]
then
 $\mathcal{P}^0$ and $\mathcal{P}^1$ have the same number $N$ of vertices $\{P^0_i\}_{i=1}^N$ and $\{P^1_i\}_{i=1}^N$, respectively, that can be ordered in such a way that
\begin{equation}\label{dve}dist(P^0_i,P^1_i)\leq C_0d_H(\der\mathcal{P}^0,\der\mathcal{P}^1) \mbox{ for every }i=1,\ldots,N.\end{equation}
More precisely $\delta_0=\min\{K_0r_0, d_0\frac{\sin\beta_0}{16}\}$ and $C_0=\sqrt{1+\frac{16}{\sin^2\beta_0}}$
\end{prop}

By Corollary \ref{coro1},  Lemma \ref{lemma3.2} and Proposition \ref{prop3.3} we have the following
\begin{prop}\label{p3.4}(Proposition 3.4 in \cite{BF}) Under the same assumptions of Theorem \ref{mainteo} there exist positive constants $\ep_0$, $\alpha$, $C>1$ depending only on the a priori data, such that, if
\begin{equation*}
	\ep:= \|\Lambda_{\mathcal{P}^0}-\Lambda_{\mathcal{P}^1}\|_*<\ep_0
\end{equation*}
then $\mathcal{P}^0$ and $\mathcal{P}^1$ have the same number $N$ of vertices
$\{P^0_j\}_{j=1}^N$ and $\{P^1_j\}_{j=1}^N$, respectively.

Moreover, the vertices can be ordered so that 
\begin{equation}\label{stimarozza}dist(P^0_j,P^1_j)\leq \omega(\ep),\quad \forall j=1,\ldots,N,\end{equation}
where $\omega(\ep)=C|\ln\ep|^{-\alpha}$. 
\end{prop}

\section{The movement from $\mathcal{P}^0$ to $\mathcal{P}^1$}
In this section we assume that
\[d_H(\der\mathcal{P}^0,\der\mathcal{P}^1)\leq \frac{\delta_0}{5},\]
where $\delta_0$ is as in Proposition \ref{prop3.3} 
and let $\{P^0_j\}_{j=1}^N$ and $\{P^1_j\}_{j=1}^N$ be the vertices of $\mathcal{P}^0$ and $\mathcal{P}^1$.
Let us also consider the intersections of $\der\mathcal{P}^0$ with $\Sigma=\{\Sigma_1,\dots,\Sigma_{m-1}\}$ and denote them by $\{P^0_j\}_{j=N+1}^{N+k}$.  

Since,  by \eqref{dve},
\[dist(P^0_i,P^1_i)\leq C_0\frac{\delta_0}{5}\leq\frac{d_0\sin\beta_0}{80}\sqrt{1+\frac{16}{\sin^2\beta_0}}\leq\frac{\sqrt{17}}{80}<\frac{d_0}{2},\,\,i=1,\dots,N\] 
by assumption  \eqref{distanzainterfaccia}, for each side of $\mathcal{P}^0$ intersecting $\Sigma$, the corresponding side in $\mathcal{P}^1$ intersects $\Sigma$.

 This means that the intersections of $\der\mathcal{P}^1$ with $\Sigma$ are given by $\{P^1_j\}_{j=N+1}^{N+k}$ and 
\[dist(P^0_j,P^1_j)\leq C_0 d_H(\der\mathcal{P}^0,\der\mathcal{P}^1),\mbox{ for } j=N+1,\ldots,N+k.\]
Let us  reorder the indices in 
\[\{P^0_j\}_{j=1}^M\mbox{ and }\{P^1_j\}_{j=1}^M\]
such that $M=N+k\leq2N$,  $P^0_jP^0_{j+1}$ is a segment in $\der\mathcal{P}^0$,
and 
\[dist(P^0_j,P^1_j)\leq C_0d_H(\der\mathcal{P}^0,\der\mathcal{P}^1) \mbox{ for }j=1,\ldots,M.\]

The aim of this section is to define a suitable  function that transform $\gamma_{\mathcal{P}_0}$ into $\gamma_{\mathcal{P}_1}$.

Let $W\subset\subset\om$ be a tubular neighborhood of  $\der\mathcal{P}^0$ of width $\frac{d_0}{4}$ so that
\[dist(W,\der\om)\geq\frac{d_0}{2}\]
and $W\supset \der\mathcal{P}^0$. 

Let us now define a  map $\mathcal{U}:\RR^2\to\RR^2$
in the following way:
Let us extend $\mathcal{U}$  to a map in $W^{1,\infty}(\RR^2)$ in such a way that
\begin{equation}\label{i.4}
supp(\mathcal{U})\subset\overline{W}\end{equation}
\begin{equation}\label{i.3}
\mathcal{U}\mbox{ is piecewise affine continuous on the line segments } P^0_jP^0_{j+1}, j=1,\dots,M, \mbox{cf.  \cite{BFV17} }, \mbox{ and on } (W\cap\Sigma)\setminus\mathcal{P}^0\end{equation} 
\begin{equation}\label{i.31}
\mathcal{U}(P^0_j)=P^1_j-P^0_j\mbox{ for every }j=1,\ldots,M.\end{equation}

\begin{equation}\label{i.5}
|\mathcal{U}|+\frac{d_0}{8} |D\mathcal{U}| \leq C_0d_H(\der\mathcal{P}^0,\der\mathcal{P}^1).
\end{equation}

\begin{prop}\label{propphi}
The map 
\[\Phi_t=I+t\mathcal{U}.\]
has the following properties:
\begin{equation}\label{f1}
\Phi_t \mbox{ is piecewise affine continuous on }\der\mathcal{P}^0\cup \Sigma\setminus\mathcal{P}^0
\end{equation}
\begin{equation}\label{f2}
\Phi_t\in W^{1,\infty}(\Omega)\mbox{ is invertible and }\end{equation}
\[|D\Phi_t-I|,\,|D\Phi^{-1}_t-I|\leq Ctd_H(\der\mathcal{P}^0,\der\mathcal{P}^1)
\]
\begin{equation}\label{f3}
\Phi_t(\om_i\setminus\mathcal{P}^0)\subset\om_i\mbox{ for all }i=1,\dots,m
\end{equation}
\begin{equation}\label{f4}
\left|\frac{d}{dt}\Phi_t\right|,\,\left|\frac{d}{dt}\Phi^{-1}_t\right|\leq Cd_H(\der\mathcal{P}^0,\der\mathcal{P}^1)
\end{equation}
\begin{equation}\label{f5}
\left|\frac{d}{dt}D\Phi_t\right|,\,\left|\frac{d}{dt}D\Phi^{-1}_t\right|\leq Cd_H(\der\mathcal{P}^0,\der\mathcal{P}^1)
\end{equation}
\begin{equation}\label{f6}
\left|\frac{d}{dt}D\Phi^{-1}_t+D\mathcal{U}\right|\leq Ctd^2_H(\der\mathcal{P}^0,\der\mathcal{P}^1),\,\,\left|\frac{d}{dt}(D\Phi^{-1}_t)^T+D\mathcal{U}^T\right|\leq Ctd^2_H(\der\mathcal{P}^0,\der\mathcal{P}^1) 
\end{equation}
\end{prop}
\begin{proof} \eqref{f1} follows from the definition of $\mathcal{U}$.

In order to prove \eqref{f2} we notice that
\[|D\Phi_t-I|=t|D\mathcal{U}|\leq 8t\frac{C_0}{d_0}d_H(\der\mathcal{P}^0,\der\mathcal{P}^1)\]
hence
\[|D\Phi_t-I|\leq
 \frac{8C_0\delta_0}{5d_0}\leq \frac{\sqrt{17}}{10}<\frac{1}{2} \]
hence $\Phi_t$ is invertible for every $t\in[0,1]$.
Moreover, by the Implicit Map Theorem it follows that $D\Phi^{-1}_t(y)=(D\Phi_t)^{-1}(\Phi_t^{-1}(y))$ and the analyticity in the parameter $t$ of $(D\Phi_t)^{-1}$ gives     
\[|D\Phi^{-1}_t-I|\leq 8t\frac{C_0}{d_0}d_H(\der\mathcal{P}^0,\der\mathcal{P}^1).\]
Let us now prove \eqref{f3} noticing that if $P^0_j\in \Sigma$, then $\mathcal{U}(P^0_j)$ is parallel to $\Sigma$, hence $\Phi_t(\Sigma)\subset \Sigma$. The invertibility of the map $\Phi_t$ then gives   \eqref{f3}.

Estimates \eqref{f4} and \eqref{f5} follows directly from  \eqref{i.5} and finally we obtain \eqref{f6} again by analyticity of $(D\Phi_t)^{-1}$ with respect to $t$:
\[
\left|\frac{d}{dt}D\Phi^{-1}_t+D\mathcal{U}\right|\leq C t|D\mathcal{U}|^2\leq C td^2_H(\der\mathcal{P}^0,\der\mathcal{P}^1)
\]
and analogously, since $(D\Phi^{-1}_t(y))^T=(D\Phi_t)^{-T}(\Phi_t^{-1}(y))$
\[
\left|\frac{d}{dt}(D\Phi^{-1}_t)^T+D\mathcal{U}^T\right|\leq C t|D\mathcal{U}|^2\leq C td^2_H(\der\mathcal{P}^0,\der\mathcal{P}^1)
\]
which concludes the proof.
\end{proof}
\section{Differentiability properties of the Dirichlet to Neumann map}
Let $\Phi_t(x)=x+t\mathcal{U}(x)$ and let us define
\[\mathcal{P}^t=\Phi_t(\mathcal{P}^0)\]
and notice that, by \eqref{f2},
\[\gamma_{\mathcal{P}^t}(x)=\gamma_{\mathcal{P}^0}\left(\Phi_t^{-1}(x)\right) .\]
Notice that the notation is consistent, that is $\gamma_{\mathcal{P}^0}\left(\Phi_1^{-1}(x)\right)=\gamma_{\mathcal{P}_1}(x)$.
Let us denote by $u_t$ the solution to 
\begin{equation}\label{t}
    \left\{\begin{array}{rcl}
             \dive(\gamma_{\mathcal{P}^t}\nabla u_t) & = & 0\mbox{ in }\om, \\
             u_t&=&f\mbox{ on }\der\om,\\
           \end{array}
    \right.
\end{equation}

Let us define \[A(t)=\left(D\Phi_t^{-1}\right)\left(D\Phi_t^{-1}\right)^T\det\left(D\Phi_t\right)\]
and
\[\mathcal{A}={\frac{d}{dt}A(t)}_{|_{t=0}}=\dive(\mathcal{U})I-(D\mathcal{U}+D\mathcal{U}^T).\]
Let us now show a result corresponding to Lemma 2.1 in \cite{BMPS}.
\begin{lm}

The solution $u_t$ to problem \eqref{t} has a material derivative at $t=0$, $\dot{u}\in H^1_0(\om)$, that solves
\[\int_\om \gamma_{\mathcal{P}^0}\nabla \dot{u}\cdot \nabla \psi=-\int_\om\gamma_{\mathcal{P}^0}\mathcal{A}\nabla u_0\cdot\nabla\psi\]
for every $\psi\in H^1_0(\om)$.
\end{lm}
\begin{proof}
Let $\tilde{f}\in H^1(\om)$ be an extension of $f$ such that $supp(\tilde{f})\cap supp(\mathcal{U})=\emptyset$ 
Let $w_t=u_t-\tilde{f}$ for $t\in[0,1]$. Notice that $w_t\in H^1_0(\om)$ and $\dive(\gamma_{\mathcal{P}^t}\nabla w_t)=-\dive(\gamma_{\mathcal{P}^t}\nabla\tilde{f})$ in $\om$, that is
\begin{equation}\label{2-6}\int_\om\gamma_{\mathcal{P}^t}\nabla w_t\cdot\nabla\psi=-\int_\om\gamma_{\mathcal{P}^t}\nabla\tilde{f}\cdot\nabla\psi\quad\forall \psi\in H^1_0(\om).\end{equation}
Let $\tilde{w}_t=w_t\circ\Phi_t$; since $\Phi_t(x)=x$ close to $\der\om$, we have that $\tilde{w}_t\in H^1_0(\om)$ satisfies
\begin{equation}\label{3-6}\int_\om\gamma_{\mathcal{P}_0}A(t)\nabla\tilde{w}_t\cdot\nabla\psi=-\int_\om\gamma_{\mathcal{P}_0}\nabla\tilde{f}\cdot\nabla\psi\quad\forall \psi\in H^1_0(\om).\end{equation}
where we have used the fact that $A=I$ on the support of $\tilde{f}$ and the fact that $\gamma_{\mathcal{P}_t}(\Phi^t(x))=\gamma_{\mathcal{P}^0}(x)$.
By subtracting \eqref{2-6} for $t=0$ from \eqref{3-6}  and dividing by $t$ we get
\begin{equation}\label{4-6}
\int_\om\gamma_{\mathcal{P}^0}A(t)\nabla\left(\frac{\tilde{w}_t-w_0}{t}\right)\cdot\nabla\psi=\int_\om\gamma_{\mathcal{P}^0}\frac{I-A(t)}{t}\nabla w_0\cdot\nabla\psi
\quad\forall \psi\in H^1_0(\om).\end{equation}


Now we proceed as in the proof of Lemma 2.1 in \cite{BMPS}. Applying \ref{propphi} we have that $\|A(t)\|_{L^{\infty}(\Omega)}\geq C$ for all $t\in [0,1]$  and that $A(t)$ is differentiable at $t=0$. Hence,  choosing 
$\psi=\frac{\tilde{w}_t-w_0}{t}\in H^1_0(\om)$ as test function and applying Cauchy-Schwarz inequality we get
\[\left\|\nabla\left(\frac{\tilde{w}_t-w_0}{t}\right)\right\|_{L^2(\om)}\leq C\|\nabla w_0\|_{L^2(\om)}\]
which means that
\[\frac{\tilde{w}_t-w_0}{t}\mbox{ is bounded in }H^1_0(\om)\]
therefore there is a weakly convergent sequence in $H^1_0(\om)$ and its weak limit is the material derivative $\dot{w}$ of $w$.
By passing to the limit in \eqref{4-6} we get that
\begin{equation}\label{2-7}\int_\om\gamma_{\mathcal{P}^0}\nabla \dot{w}\cdot\nabla\psi=-\int_\om\gamma_{\mathcal{P}^0}\mathcal{A}\nabla w_0\cdot\nabla\psi\quad\forall \psi\in H^1_0(\om).\end{equation}
We also have strong convergence. In fact,  if we take $\psi=\frac{\tilde{w}-w_0}{t}$ in \eqref{4-6}  we get
\begin{eqnarray*}&&\int_\om\gamma_{\mathcal{P}^0}A(t)\nabla\left(\frac{\tilde{w}_t-w_0}{t}\right)\cdot\nabla\left(\frac{\tilde{w}_t-w_0}{t}\right)
=\int_\om\gamma_{\mathcal{P}^0}\frac{I-A(t)}{t}\nabla w_0\cdot\nabla \left(\frac{\tilde{w}_t-w_0}{t}\right)\\
\end{eqnarray*}
and, using weak convergence of $\frac{\tilde{w}_t-w_0}{t}$ in $H^1(\Omega$) strong convergence of $\frac{I-A(t)}{t}\nabla w_0$ in $L^2(\Omega)$ and by  \eqref{2-7} where we have chosen $\dot{w}$ as test function, we obtain
\[\lim_{t\to 0}\int_\om\gamma_{\mathcal{P}^0}A(t)\nabla\left(\frac{\tilde{w}_t-w_0}{t}\right)\cdot\nabla\left(\frac{\tilde{w}_t-w_0}{t}\right)=-\int_\om\gamma_{\mathcal{P}^0}\mathcal{A}\nabla w_0\cdot\nabla\dot{w}=\int_\om\gamma_{\mathcal{P}^0}\nabla \dot{w}\cdot\nabla\dot{w}.\]
Finally, by
\begin{eqnarray*}
\int_\om\gamma_{\mathcal{P}^0}\left|\nabla\left(\frac{\tilde{w}_t-w_0}{t}\right)\right|^2&=&\int_\om\gamma_{\mathcal{P}^0}A(t)\nabla\left(\frac{\tilde{w}_t-w_0}{t}\right)\cdot\nabla\left(\frac{\tilde{w}_t-w_0}{t}\right)\\&&+\int_\om\gamma_{\mathcal{P}^0}(I-A(t))\nabla\left(\frac{\tilde{w}_t-w_0}{t}\right)\cdot\nabla\left(\frac{\tilde{w}_t-w_0}{t}\right)
\end{eqnarray*}
and since
\[
\int_\om\gamma_{\mathcal{P}^0}(I-A(t))\nabla\left(\frac{\tilde{w}_t-w_0}{t}\right)\cdot\nabla\left(\frac{\tilde{w}_t-w_0}{t}\right)\leq Ct
\]
we obtain 
\[\left\|\sqrt{\gamma_{\mathcal{P}^0}}\nabla\left(\frac{\tilde{w}_t-w_0}{t}\right)\right\|^2_{L^2(\Omega)}\rightarrow \|\sqrt{\gamma_{\mathcal{P}^0}}\nabla\dot{w}\|^2_{L^2(\Omega)}.
\]
This last convergence with the weak convergence in $H^1(\Omega)$ and coercivity finally gives 
\[
\left\|\nabla\left(\frac{\tilde{w}_t-w_0}{t}-\dot{w}\right)\right\|_{L^2(\Omega)}\rightarrow 0
\]
as $t\rightarrow 0$.
Finally, by Poincar\'e inequality, $\frac{\tilde{w}_t-w_0}{t}\to \dot{w}$ in $H^1_0(\om)$.

Let us now go back to $u_t$. Since $u_t=w_t+\tilde{f}$ and $\tilde{u}_t=\tilde{w}_t+\tilde{f}(\Phi_t)$, we have
\[\frac{\tilde{u}_t-u_0}{t}=\frac{\tilde{w}_t-w_0}{t}+\frac{\tilde{f}(\Phi_t)-\tilde{f}}{t}.\]
But, we notice that $\tilde{f}(\Phi_t)-\tilde{f}\equiv 0$ in $\om$, since $supp(\tilde{f})\subset\{x\,:\,\Phi_t(x)=x\}$, hence $\dot{u}=\dot{w}$. 
\end{proof}
\subsection{Derivative of the forward map}
Let us now evaluate the Gateaux derivative of the DN map along the direction of the vector field $\mathcal{U}$. 

Let $f,g\in H^{1/2}(\der\om)$ and let $u_t$ solve \eqref{t} and $v_t$ be the unique solution to
\begin{equation}\label{vt}
    \left\{\begin{array}{rcl}
             \dive(\gamma_{\mathcal{P}^t}\nabla v_t) & = & 0\mbox{ in }\om, \\
             v_t&=&g\mbox{ on }\der\om.

           \end{array}
    \right.
\end{equation}
Let 
\begin{eqnarray*}F(t,f,g)&=&<\Lambda_{\mathcal{P}^t}f,g>=\int_\om\gamma_{\mathcal{P}^t}\nabla u_t \cdot \nabla v_t=<\gamma_b\frac{\der u_t}{\der n},g>
\end{eqnarray*}
and analogoulsly let
\[F(0,f,g)=<\Lambda_{\mathcal{P}_0}f,g>=\int_\om\gamma_{\mathcal{P}^0}\nabla u_0 \cdot \nabla v_0=<\gamma_b\frac{\der u_0}{\der n},g>.
\]
Hence,
\[
\frac{F(t,f,g)-F(0,f,g)}{t}=<\gamma_b\frac{\der }{\der n}\left(\frac{u_t-u_0}{t}\right),g>.
\]
Since $\gamma_b\frac{\der u_t}{\der n}=\gamma_b\frac{\der \tilde{u}_t}{\der n}$ because $supp(\mathcal{U})$ is far from $\der\om$, and by 
\[\frac{\tilde{u}_t-u_0}{t}\to\dot{u}\]
in $H^1(\om)$, we get
\[\frac{F(t,f,g)-F(0,f,g)}{t}\to\,\, <\gamma_{\mathcal{P}^0}\frac{\der \dot{u}}{\der n},g>
=\int_\om\gamma_{\mathcal{P}^0}\nabla\dot{u}\cdot\nabla v_0=-\int_\om\gamma_{\mathcal{P}^0}\mathcal{A}\nabla u_0\cdot\nabla v_0\]
as $t\rightarrow 0$. Hence
\[F^\prime(0,f,g)=-\int_\om\gamma_{\mathcal{P}^0}\mathcal{A}\nabla u_0\cdot\nabla v_0,\]
where we denote by $F^\prime$ the derivative with respect to $t$.

\begin{rem}
Using a similar argument, one can show that $F(t,f,g)$ is differentiable for any $t_0\in [0,1]$ and that 
\[F^\prime(t_0,f,g)=-\int_\om\gamma_{\mathcal{P}^{t_0}}\mathcal{A}_{t_0}\nabla u_{t_0}\cdot\nabla v_{t_0}\]
for 
\[\mathcal{A}_{t_0}=\frac{d}{dt}\left(D\Phi_{t_0,t}^{-1}\right)\left(D\Phi_{t_0,t}^{-1}\right)^T\det\left(D\Phi_{t_0,t}\right)|_{t=t_0}\]
where $\Phi_{t_0,t}=I+t\mathcal{U}_{t_0}$ and  $\mathcal{U}_{t_0}$ is a $W^{1,\infty}(\Omega)$ map satisfying (\ref{i.4}), (\ref{i.5}) and such that $\mathcal{U}_{t_0}(P_j^0+t_0(P_j^1-P^0_j))=P_j^1-P^0_j\,\, \forall j=1, \dots, M$ and finally, $u_{t_0},v_{t_0}$ are solutions of (\ref{t}), (\ref{vt}) respectively for conductivity $\gamma_{\mathcal{P}_{t_0}}$.

\end{rem}
\subsection{Continuity of the Gateaux derivative}
We now want to prove the following result
\begin{prop}
There exist constants $C$, $\beta>0$ depending only on the a priori data, such that for all $t\in [0,1]$
\begin{equation}\label{contder}\left|F'(t,f,g)-F'(0,f,g)	\right|\leq C\|f\|_{H^{1/2}(\der\om)}\|g\|_{H^{1/2}(\der\om}t^\beta d^{1+\beta}_H(\der\mathcal{P}^0,\der\mathcal{P}^1).\end{equation}

\end{prop}
\begin{proof}
To simplify the notation throughout the proof we will use  $d_H:=d_{H}\left(\der \mathcal{P}_0,\der\mathcal{P}_1\right)$ 
\begin{eqnarray*}
F'(t,f,g)-F'(0,f,g)&=&-\int_\om\gamma_{\mathcal{P}^t}\mathcal{A}_t\nabla u_t\cdot\nabla v_t+\int_\om\gamma_{\mathcal{P}^0}\mathcal{A}\nabla u_0\cdot\nabla v_0\\
&=&-\int_\om(\gamma_{\mathcal{P}^t}-\gamma_{\mathcal{P}^0})\mathcal{A}\nabla u_0\cdot\nabla v_0-\int_\om\gamma_{\mathcal{P}^t}(\mathcal{A}_t-\mathcal{A})\nabla u_t\cdot\nabla v_t-\\
&-&\int_\om\gamma_{\mathcal{P}^t}\mathcal{A}\nabla (u_t-u_0)\cdot\nabla v_0-\int_\om\gamma_{\mathcal{P}^t}\mathcal{A}\nabla u_t\cdot\nabla (v_t-v_0)\\
&=&I_1+I_2+I_3+I_4
\end{eqnarray*}
Let us start estimating $I_1$
\begin{eqnarray}\label{sum}
|I_1|&=&\left|\int_{\mathcal{P}^t\triangle\mathcal{P}^0}(k-\gamma_b)\mathcal{A}\nabla u_0\cdot\nabla v_0\right|\\
&\leq& C\|\mathcal{A}\|_{L^{\infty}(\Omega)}\|\nabla u_0\|_{L^2(\mathcal{P}^t\triangle\mathcal{P}^0)}\|\nabla v_0\|_{L^2(\mathcal{P}^t\triangle\mathcal{P}^0)}.
\end{eqnarray}
By Meyer's theorem, $u_0,v_0\in W^{1,p}(\om)$ with $p>2$ so that, applying H\"older inequality we get
\[
\|\nabla u_0\|^2_{L^2(\mathcal{P}^t\triangle\mathcal{P}^0)}\leq |\mathcal{P}^t\triangle\mathcal{P}^0|^{1-\frac{2}{p}}\|\nabla u_0\|^{\frac{2}{p}}_{L^p(\om)}
\]
which implies
\begin{equation}\label{gradu0}
\|\nabla u_0\|_{L^2(\mathcal{P}^t\triangle\mathcal{P}^0)}\leq Ct^{\frac{1}{2}-\frac{1}{p}}d^{\frac{1}{2}-\frac{1}{p}}_H\|\nabla u_0\|_{L^p(\om)}
\leq Ct^{\frac{1}{2}-\frac{1}{p}}d^{\frac{1}{2}-\frac{1}{p}}_H\|f\|_{H^{1/2}(\der\om)}
\end{equation}
and analogously 
\begin{equation}\label{gradv0}
\|\nabla v_0\|_{L^2(\mathcal{P}^t\triangle\mathcal{P}^0)}\leq Ct^{\frac{1}{2}-\frac{1}{p}}d^{\frac{1}{2}-\frac{1}{p}}_H\|g\|_{H^{1/2}(\der\om)}.
\end{equation}
Moreover, by (\ref{i.5}) $\|\mathcal{A}\|_{L^{\infty}(\om)}\leq Cd_H$.
Hence,
\begin{equation}\label{I1}
|I_1|\leq Ct^{1-\frac{2}{p}}d^{2-\frac{2}{p}}_H.
\end{equation}
For estimating $I_2$ we note that by (\ref{i.5}) and Proposition \ref{propphi} we have that \begin{equation}\label{At}
\|\mathcal{A}_t-\mathcal{A}\|_{L^{\infty}(\om)}\leq Ct d^{2}_H.
\end{equation}
In fact, let us rewrite
\begin{eqnarray*}
\mathcal{A}_t-\mathcal{A}&=&\frac{d}{dt}det D\Phi_t D\Phi_t^{-1}D\Phi_t^{-T}-div\,\mathcal{U}I\\
&+&det D\Phi_t \left(\frac{d}{dt}D\Phi^{-1}_t\right)D\Phi_t^{-T}+D\mathcal{U}+det D\Phi_t D\Phi^{-1}_t\left(\frac{d}{dt}D\Phi_t^{-T}\right)+D\mathcal{U}^T\\
&=&J_1+J_2+J_3.
\end{eqnarray*}
From Proposition \ref{propphi} and applying (\ref{i.5}) we have the following bound on $J_1$ and we get
\begin{equation}\label{J1}
|J_1|\leq C\left(\left|\frac{d}{dt}det D\Phi_t-div\,\mathcal{U}I\right|+div\,\mathcal{U}(|D\Phi^{-1}_t-I|+|D\Phi^{-T}_t-I|)\right)\leq Ctd^{2}_H
\end{equation}
and again by Proposition\ref{propphi} we derive
\begin{equation}\label{J2}
|J_2|\leq C(\|D\mathcal{U}\|_{L^{\infty}(\om)}|D\Phi^{-T}_t-I|+|\frac{d}{dt}D\Phi^{-1}_t+D\mathcal{U}|)\leq C td^{2}_H
\end{equation}
and arguing as for $J_2$
\begin{equation}\label{J3}
|J_3|\leq C td^{2}_H.
\end{equation}
Collecting (\ref{J1}),(\ref{J2}) and (\ref{J3}) we have (\ref{At}) which gives
\begin{equation}\label{I2}
|I_2|\leq Ctd^{2}_H
\|f\|_{H^{1/2}(\der\om)}\|g\|_{H^{1/2}(\der\om)}.
\end{equation}
For the last two terms $I_3,I_4$ we note that
\begin{equation}\label{I3}
|I_3|\leq C \|\mathcal{A}\|_{L^{\infty}(\om)}\|u_t-u_0\|_{L^2(\om)}\|v_0\|_{L^2(\om)}\leq Ct^{\frac{1}{2}-\frac{1}{p}}d^{\frac{3}{2}-\frac{1}{p}}_H\|f\|_{H^{1/2}(\der\om)}\|g\|_{H^{1/2}(\der\om)}
\end{equation}
where we have used the estimate
\[
\|u_t-u_0\|_{L^2(\om)}\leq C \|\nabla u_0\|_{L^2(\mathcal{P}^t\triangle\mathcal{P}^0)}
\]
and the bound (\ref{gradu0}).
Analogously,
\begin{equation}\label{I4}
|I_4|\leq C \|\mathcal{A}\|_{L^{\infty}(\om)}\|v_t-v_0\|_{L^2(\om)}\|u_0\|_{L^2(\om)}\leq Ct^{\frac{1}{2}-\frac{1}{p}}d^{\frac{3}{2}-\frac{1}{p}}_H\|f\|_{H^{1/2}(\der\om)}\|g\|_{H^{1/2}(\der\om)}
\end{equation}
Finally, collecting (\ref{I1}), (\ref{I2}),(\ref{I3}),(\ref{I4}) and setting $\beta=\frac{1}{2}-\frac{1}{p}>0$ the claim follows.

\end{proof}
\subsection{Bound from below for the derivative}
In order to simplify the exposition of the proof of the lower bound for the derivative we will assume that $\Omega$ contain only one interior layer $\Sigma=(-L,L)\times\{0\}$  and we denote with $\om_-=(-L,L)\times(-L,0)$ and $\om_+=(-L,L)\times(0,L)$ the two layers with corresponding conductivity $\gamma_-$ and $\gamma_+$. 
In order to perform the estimate from below of the derivative of $F$ we need to state some regularity result for solutions of the equation in stratified media. The proposition below is a special case of Proposition 1.6 in \cite{LN}. 



\begin{prop}\label{LiNir}
Let $B_r$ be the disk of radius $r>0$ centered at the origin, and let $B_r^\pm$ be the upper and the lower half disk and let $\gamma^1$ and $\gamma^2$ be two positive constants.

Let $v\in H^1(B_r)$ be a solution to
\begin{equation}\label{sal}\dive \left(\left(\gamma^1+(\gamma^2-\gamma^1)\chi_{B_r^+}\right)\nabla v\right)=0 \mbox{ in }B_r.
\end{equation}
Then $v\in C^\infty\left(\overline{B_r^+}\right)\cap C^\infty\left(\overline{B_r^-}\right)$ and for every $\delta>0$ there is a constant $C$ depending only on $\gamma^1$, $\gamma^2$ and $\delta$ such that
\begin{equation}\label{linir}
    \|\nabla v\|_{L^\infty(B_{(1-\delta)r})}\leq C\|v\|_{L^2(B_r)}.
\end{equation}
\end{prop}

Let 
\[\Omega_0=\Omega\cup \left\{(-d_0,d_0)\times [L,L+2d_0]\right\}\]
and let us extend $\gamma_{\mathcal{P}^0}$ to $\Omega_0$ by setting $\gamma_{\mathcal{P}^0}=\gamma_+$ in $(-d_0,d_0)\times [L,L+2d_0]$.

We now state some useful estimates of  the Green function $G_0(x,y)$ corresponding to the operator $\dive\left(\gamma_{\mathcal{P}^0}\nabla\cdot\right)$ and to the domain $\Omega_0$.
Let $y\in\Omega_0\backslash\mathcal{P}^0$ and let $0<r<\text{dist}(y,\{P_i^0\}_{j=1}^M\cup\partial \Omega_0)$. Then, in the ball $B_{r}(y)$ either $\gamma_{\mathcal{P}^0}$ is constant or $\gamma_{\mathcal{P}^0}=\gamma_b$,
for a  suitable choice of the coordinate system,
$\gamma_{\mathcal{P}^0}=\gamma_{+} +(k-\gamma_{+})\chi_{\{x_2>a\}} $ or $\gamma_{\mathcal{P}^0}=\gamma_{-} +(k-\gamma_{-})\chi_{\{x_2>a\}} $ for some some $a$ with $|a|<r$.
Let
\[
   \\
   \gamma_y= \left\{\begin{array}{rcl}
            &\gamma_{+} \text{ or }\gamma_{-} \quad&\text{ if }\gamma_{\mathcal{P}^0}=\gamma_{+}\text{ or }\gamma_{-}\text{ in }B_{r}(y),\\
             &\gamma_b\quad&\text{ if }\gamma_{\mathcal{P}^0}=\gamma_b,\text{ in }B_{r}(y),\\ 
             &\gamma_{+} +(k-\gamma_{+})\chi_{\{x_2>a\}}\text{ or }\gamma_{-} +(k-\gamma_{-})\chi_{\{x_2>a\}} & \text{ otherwise},
           \end{array}\right.
\]

and consider the bi-phase fundamental solution to
\[
 \dive\left(\gamma_y\nabla\overline{\Gamma}(\cdot, y)\right)=\delta_y\text{ in }\RR^2.
\]
\begin{prop}\label{Green}
There exists a constant $C>0$ depending only on the a-priori data such that 
for $y\in\Omega_0\backslash \mathcal{P}^0$ and $dist(y,\{P_i^0\}_{j=1}^M\cup\der\Omega_0)\geq d_0/c_1$ for some $c_1>1$,
\begin{equation}\label{estGreen2}
\|G_0(\cdot,y)-\overline{\Gamma}(\cdot, y)\|_{H^1(\om_0)}\leq C
\end{equation}
and 
\begin{equation}\label{energyGreen}
\|G_0(\cdot,y)\|_{H^1(\Omega_0\backslash B_{r}(y))}\leq C\left|\ln\frac{D}{r}\right|^{1/2}
\end{equation}
where $D$ depends only on $L$ and $d_0$.\\
Furthermore, let $y_r=P+rn(P)$, where $P$ is a point on $\der\mathcal{P}^0$ such that $dist(P,\{P_i^0\}_{j=1}^M)\geq d_0/c_1$ for some $c_1>1$ and $n(P)$ is the unit outer normal to $\der\mathcal{P}^0$. 
Then for $r$ small enough and for $x\in\mathcal{P}^0\cap B(P,d_0/2c_1)$ , we have
\begin{equation}\label{estGreen3}
	\left |\nabla G_0(x,y_r)-\nabla\overline{\Gamma}(x,y_r)\right|\leq C.
\end{equation}
where $$\nabla\overline{\Gamma}(x,y_r)=\frac{2}{\gamma_++k}\nabla\Gamma(x,y_r)$$ if $B(P,d_0/2c_1)$ intersects $\Omega^+$  or $$\nabla\overline{\Gamma}(x,y_r)=\frac{2}{\gamma_-+k}\nabla\Gamma(x,y_r)$$ if $B(P,d_0/2c_1)$ intersects $\Omega^-$ and where $\Gamma(x,y)$ is the fundamental solution for the Laplacian operator.

\end{prop}
\begin{proof}
The proof of (\ref{estGreen2}) can be derived with similar arguments as in Proposition 3.4 in \cite{BF11}.  Observe that from (\ref{estGreen2}) we have
\begin{equation}\label{estimate1}
   \|G_0(\cdot,y)\|_{H^1(\Omega_0\backslash B_{r}(y))}\leq
 \|G_0(\cdot,y)-\overline{\Gamma}(\cdot, y)\|_{H^1(\om_0)}+\|\overline{\Gamma}(\cdot,y)\|_{H^1(\Omega_0\backslash B_{r}(y))}
 \leq C+\|\overline{\Gamma}(\cdot,y)\|_{H^1(\Omega_0\backslash B_{r}(y))}.
 \end{equation}
 Now, 
\begin{equation}\label{estimate2}
\|\overline{\Gamma}(\cdot,y)\|_{H^1(\Omega_0\backslash B_{r}(y))}\leq 
\|\overline{\Gamma}(\cdot,y)\|_{H^1((B_{D}(y)\backslash B_{r}(y)))\cap\Omega_0}
\end{equation}
where $D=\sqrt{4L^2+(2L+2d_0)^2}$. Now, using the explicit representation for the bi-phase fundamental solution (see for example (4.26) in \cite{AV}) we have that 
\[
|\overline{\Gamma}(x,y)|\leq C|\ln|x-y||,\,\, |\nabla\overline{\Gamma}(x,y)|\leq C|x-y|^{-1}
\]
where $C$ depends on $\gamma^+$ and $\gamma^-$ and an easy computation shows that 
\begin{equation}\label{Fundamental}
\|\overline{\Gamma}(\cdot,y)\|_{H^1((B_{D}(y)\backslash B_{r}(y)))\cap\Omega_0}\leq C\left|\ln\frac{D}{r}\right|^{1/2}.
\end{equation}
Hence, from last inequality, (\ref{estimate2}) and (\ref{estimate1})
(\ref{energyGreen}) follows.
Finally, (\ref{estGreen3}) follows from similar arguments as in Proposition 3.4 in \cite{BF11}.  
\end{proof}
\begin{prop}\label{p3.3}
There exist a constant $m_0>0$, depending only on the a priori data, and a pair of functions $f_0$ and $g_0$ in $H^{1/2}(\der\om)$ such that
\begin{equation}\label{tesibasso}\left|{F^\prime(0,f_0,g_0)}\right|\geq m_0 d_H\|f_0\|_{H^{1/2}(\der\om)}\|g_0\|_{H^{1/2}(\der\om)}. \end{equation}
\end{prop}
\begin{proof}
Let us set 
\[V=\left(P^1_1-P^0_1,P^1_2-P^0_2,\ldots,P^1_M-P^0_M\right).\]
We recall that
\[C^{-1}d_H\leq |V|\leq d_H\]
for $C$ depending only on the a priori data.

Let us first normalize the length of vector $|V|$ by setting
\[\tilde{\mathcal{U}}=\frac{\mathcal{U}}{|V|}, \quad \tilde{\mathcal{A}}=\frac{\mathcal{A}}{|V|}\]
and
\[H(f,g)=-\int_\om\gamma_{\mathcal{P}^0}\tilde{\mathcal{A}}\nabla u_0\cdot\nabla v_0,\]
so that 
\[F^\prime(0,f,g)=|V|\, H(f,g).\]
Let 
\[m_1=\|H\|_*=\sup\left\{\frac{|H(f,g)|}{\|f\|_{H^{1/2}(\der\om)}\|g\|_{H^{1/2}(\der\om)}}\,:\,f,g\neq 0\right\}\] be the operator norm of $H$, so that
\begin{equation}\label{normaH}\left|H(f,g)\right|\leq m_1\|f\|_{H^{1/2}(\der\om)}\|g\|_{H^{1/2}(\der\om)}\mbox{ for every }f,g\in H^{1/2}(\der\om).
\end{equation}



Let us now set, for $y,z\in\Omega_0\setminus\overline{\om}$,
\begin{equation}\label{gr}
u_0(x)=G_0(x,y)\quad\mbox{ and }\quad v_0(x)=G_0(x,z)
\end{equation}
and define 
\[S_0(y,z)=-\int_\Omega\gamma_{\mathcal{P}^0}\tilde{\mathcal{A}}\nabla u_0\cdot\nabla v_0.\]
Notice that, for $y,z\in D_0=[-\frac{d_0}{2},\frac{d_0}{2}]\times [L+\frac{d_0}{2},L+\frac{3d_0}{2}]$
\[S_0(y,z)=H({G_0(\cdot,y)}_{|_{\der\Omega}},{G_0(\cdot,y)}_{|_{\der\Omega}})\]
hence, by \eqref{normaH}, by  (\ref{energyGreen}) of  Proposition \ref{Green} with $r=d_0/4$, and, finally, by the Trace Theorem we have 
\begin{equation}\label{12.1}
\left|S_0(y,z)\right|\leq C m_1\left|\ln\frac{4D}{d_0}\right|\mbox{ for }y,z\in D_0,
\end{equation}
where $C$ depends on the a-priori data.

By assumptions \eqref{lati} and \eqref{distanzainterfaccia}, there exists a constant $C_0$ depending only on the a-priori data, such that
\[dist(P^0_j,P^0_l)>\frac{2d_0}{C_0}\mbox{ if }j\neq l\]
and 
\[B\left(P^0_j,\frac{d_0}{C_0}\right)\mbox{ does not intersect sides of }\mathcal{P}^0\mbox{ that do not contain }P^0_j\]
Let $\mathcal{B}=\cup_{j=1}^{M}B\left(P^0_j,\frac{d_0}{4C_0}\right)$ and let us write
\begin{equation}\label{e1}
S_0(y,z)=-\int_{\Omega\setminus\mathcal{B}}\gamma_{\mathcal{P}^0}\tilde{\mathcal{A}}\nabla u_0\cdot\nabla v_0-\int_\mathcal{B}\gamma_{\mathcal{P}^0}\tilde{\mathcal{A}}\nabla u_0\cdot\nabla v_0.
\end{equation}
Since $\gamma_{\mathcal{P}^0}$ is piece-wise constant, the solutions $u_0$ and $v_0$ are harmonic in each domain where $\gamma_{\mathcal{P}^0}$ is constant, that is in $\Omega_+\setminus{\mathcal{P}^0}$, $\Omega_-\setminus{\mathcal{P}^0}$ and $\mathcal{P}^0$.
In each of this  sets
\[\tilde{\mathcal{A}}\nabla u_0\cdot\nabla v_0=-\dive(b)\]
where 
\[b=\left(\tilde{\mathcal{U}}\cdot\nabla u_0\right)\nabla v_0+\left(\tilde{\mathcal{U}}\cdot\nabla v_0\right)\nabla u_0-\left(\nabla v_0\cdot\nabla u_0\right)\tilde{\mathcal{U}},\]
(see \cite{BMPS} for details).
We can write
\begin{equation}\label{e2.0}
\int_{\Omega\setminus\mathcal{B}}\!\!\!\!\gamma_{\mathcal{P}^0}\tilde{\mathcal{A}}\nabla u_0\cdot\nabla v_0= -\gamma_+\int_{\Omega_+\setminus(\mathcal{P}^0\cup\mathcal{B})}\!\!\!\!\!\!\!\!\!\!\!\!\dive(b)
- \gamma_-\int_{\Omega_-\setminus(\mathcal{P}^0\cup\mathcal{B})}\!\!\!\!\!\!\!\!\!\!\!\!\dive(b)- k\int_{\mathcal{P}^0\setminus\mathcal{B}}\!\!\!\!\!\!\!\!\!\!\dive(b).
\end{equation}

For any function $u$ defined in $\om$, let us use the following notation:
\[u^+=u_{|_{\overline{\Omega}_+\setminus\mathcal{P}_0}},\quad u^-=u_{|_{\overline{\Omega}_-\setminus\mathcal{P}_0}}, \quad u^i=u_{|_{\overline{\mathcal{P}}_0}}.
\]
Let us now denote by $n$ the unit outer normal to $\der\om$ and to $\der\mathcal{P}^0$, by $r$ the unit outer normal to the balls in $\mathcal{B}$ and by $e_2$ the unit vertical normal vector.

Integrating by parts and recalling that $supp(\tilde{\mathcal{U}})\subset W$, hence $b=0$ on $\der\om$, we have
\begin{equation}\label{e2.1}
    \int_{\Omega_+\setminus(\mathcal{P}^0\cup\mathcal{B})}\!\!\!\!\!\!\!\!\!\!\!\!\dive(b^+) =
    -\int_{(\der\mathcal{B}\cap\Omega_+)\setminus\mathcal{P}^0}\!\!\!\!\!\!\!\!\!\!\!\! b^+\cdot r - \int_{(\der\mathcal{P}^0\cap\Omega_+)\setminus\mathcal{B}}\!\!\!\!\!\!\!\!\!\!\!\! b^+\cdot n-\int_{\Sigma\setminus\mathcal{P}^0}\!\!\!\!\!\!\!\! b^+\cdot e_2,
\end{equation}
\begin{equation}\label{e2.2}
    \int_{\Omega_-\setminus(\mathcal{P}^0\cup\mathcal{B})}\!\!\!\!\!\!\!\!\!\!\!\!\dive(b^-) =
    -\int_{(\der\mathcal{B}\cap\Omega_-)\setminus\mathcal{P}^0}\!\!\!\!\!\!\!\!\!\!\!\! b^-\cdot r - \int_{(\der\mathcal{P}^0\cap\Omega_-)\setminus\mathcal{B}}\!\!\!\!\!\!\!\!\!\!\!\! b^-\cdot n+\int_{\Sigma\setminus\mathcal{P}^0}\!\!\!\!\!\!\!\! b^-\cdot e_2
\end{equation}
and
\begin{equation}\label{e2.3}
    \int_{\mathcal{P}^0\setminus\mathcal{B}}\!\!\!\!\!\dive(b^i) =
    \int_{\der\mathcal{P}^0\setminus\mathcal{B}}\!\!\!\!\!\!\!\!\! b^i\cdot n-\int_{\der\mathcal{B}\cap\mathcal{P}^0}\!\!\!\!\!\!\!\!\!\! b^i\cdot r.
\end{equation}
By \eqref{e2.0}, \eqref{e2.1}, \eqref{e2.2} and \eqref{e2.3} we get
\begin{equation}\label{e2}
\int_{\Omega\setminus\mathcal{B}}\!\!\!\!\gamma_{\mathcal{P}^0}\tilde{\mathcal{A}}\nabla u_0\cdot\nabla v_0=\int_{\der\mathcal{B}}\gamma_{\mathcal{P}^0}b\cdot r+\int_{\Sigma\setminus\mathcal{P}^0}[ \gamma_{\mathcal{P}^0}b\cdot e_2]-
\int_{\der\mathcal{P}^0\setminus\mathcal{B}}[\gamma_{\mathcal{P}^0}b\cdot n]
\end{equation}
where $[\quad]$ denotes the jump.

By definition,  we have $\tilde{\mathcal{U}}\cdot e_2=0$ on $\Sigma$, and, by transmission conditions for $u_0$ and $v_0$ on $\Sigma$, we have
$\gamma_+\nabla u_0^+\cdot e_2=\gamma_-\nabla u_0^-\cdot e_2$ and $\gamma_+\nabla v_0^+\cdot e_2=\gamma_-\nabla v_0^-\cdot e_2$ while $\tilde{\mathcal{U}}\cdot \nabla u_0^+=\tilde{\mathcal{U}}\cdot \nabla u_0^-$ and  $\tilde{\mathcal{U}}\cdot \nabla v_0^+=\tilde{\mathcal{U}}\cdot \nabla v_0^-$. For these reasons,

\begin{equation}\label{b51}
[ \gamma_{\mathcal{P}^0}b\cdot e_2]=0\mbox{ on }\Sigma.
\end{equation}
In a similar way, by transmission conditions on $\der\mathcal{P}^0$, we can write
\begin{equation}\label{b52}
[\gamma_{\mathcal{P}^0}b\cdot n]= \left(\tilde{\mathcal{U}}\cdot n\right)(k-\gamma_+)\mathcal{M}^+ \nabla u_0^i\cdot \nabla v_0^i \mbox{ on }\der\mathcal{P}^0\setminus\mathcal{B}\cap \om_+
\end{equation}
where  $\mathcal{M}^+$ is a tensor with eigenvectors $n$ and $n^\perp$ with eigenvalues $\frac{k}{\gamma_+}$ and $1$, and
 \begin{equation}\label{b61}
[\gamma_{\mathcal{P}^0}b\cdot n]= \left(\tilde{\mathcal{U}}\cdot n\right)(k-\gamma_-)\mathcal{M}^- \nabla u_0^i\cdot \nabla v_0^i \mbox{ on }\der\mathcal{P}^0\setminus\mathcal{B}\cap \om_-
\end{equation}
where $\mathcal{M}^-$ is a tensor with eigenvectors $n$ and $n^\perp$ with eigenvalues $\frac{k}{\gamma_-}$ and $1$.

Hence, finally, by \eqref{e1}, \eqref{e2}, \eqref{b51}, \eqref{b52} and \eqref{b61},

\begin{equation}\label{b62}
S_0(y,z)=-\int_{\mathcal{B}}\gamma_{\mathcal{P}^0}\tilde{\mathcal{A}}\nabla u_0\nabla v_0-\int_{\der\mathcal{B}}\gamma_{\mathcal{P}^0}b\cdot r
+\int_{\der\mathcal{P}^0\setminus\mathcal{B}}\left(\tilde{\mathcal{U}}\cdot n\right)(k-\gamma_{\mathcal{P}^0})\mathcal{M} \nabla u_0^i\cdot \nabla v_0^i 
\end{equation}
where $\mathcal{M}=\mathcal{M}^+\chi_{\om_+\cap {\der\mathcal{P}^0}}+
\mathcal{M}^-\chi_{\om_-\cap {\der\mathcal{P}^0}}$.
 
From formula \eqref{b62} we deduce that $S_0$ is well defined for $y,z\in \om_0\setminus \overline{\mathcal{P}^0\cup \mathcal{B}}$ and, recalling \eqref{gr}, we have that $S_0$ solves the equation
\begin{equation}\label{diffequS}
\dive\left(\gamma_{\mathcal{P}^0}\nabla S_0\right)=0 \mbox{ in }\om_0\setminus \overline{\mathcal{P}^0\cup \mathcal{B}}
\end{equation}
both with respect to $y$ and to $z$.

Let us now consider any segment $P^0_jP^0_{j+1}$ of $\der\mathcal{P}^0$ and let $P$ be the mid-point of such segment.

By assumptions on \eqref{angoli}-\eqref{distanzainterfaccia} 
the disk $B(P,\frac{d_0}{C_0})$ intersects $\der\mathcal{P}^0$ only on the side containing $P^0_jP^0_{j+1}$ and it does not intersect $\Sigma$.

Let $\sigma$ be a simple curve joining the point $P+\frac{d_0}{C_0}n(P)$ to the point $Q_0=(0,L+d_0)$ such that $\sigma\subset\om$ and $dist(\sigma,\mathcal{P}^0)>\frac{d_0}{C_0}$. Let $\sigma^\prime=\sigma\cup\{P+tn(P)\,:\,t\in [0,\frac{d_0}{C_0}]\}$.


Let $K=\left\{x\in\om_0\setminus\mathcal{P}_0\, :\, dist(x,\sigma^\prime)<\frac{d_0}{4C_0}\right\}$ and $K_0=\left\{x\in\om_0\setminus\mathcal{P}_0\, :\, dist(x,\sigma^\prime)<\frac{d_0}{8C_0}\right\}$


The function $S_0$ solves in $K$ the equation $\dive\left(\gamma_{\mathcal{P}^0}\nabla S_0\right)=0$ with respect to both $y$ and  $z$.
Let us now estimates function $S_0(y,z)$ for points $y,z\in K$.

First of all
by \eqref{energyGreen} we notice that, since  $dist(K,\mathcal{B})
\geq \tilde{r}_0:=\frac{d_0}{4C_0}$, then
\[ \|\nabla u_0\|_{L^2(\mathcal{B})}\leq \|G_0(\cdot,y)\|_{H^1(\om_0\setminus B_{\tilde{r}_0}(y))}\leq C\]
where $C$ depends only on the apriori constants. Since
the same holds for $v_0$, we have
\begin{equation}\label{s1}\left|\int_{\mathcal{B}}\gamma_{\mathcal{P}^0}\tilde{\mathcal{A}}\nabla u_0\nabla v_0\right|\leq C \|\nabla u_0\|_{L^2(\mathcal{B})}\|\nabla v_0\|_{L^2(\mathcal{B})}
\leq C.
\end{equation}

Let us now notice that we can cover $\der\mathcal{B}$ by a finite number of balls of radius $\frac{d_0}{16 C_0}$ such that, in each  ball of radius $\frac{d_0}{8 C_0}$ concentric to them, the function $u_0$ solves an 
equation of the form \eqref{sal} with $h=g=0$ (for some balls $\gamma^1=\gamma^2$).

By combining \eqref{linir} with  \eqref{energyGreen} and the fact that the distance between $K$ and the largest balls covering $\der\mathcal{B}$  is bigger than $\tilde{r}_0$, then
$\|\nabla u_0\|_{L^\infty(\der\mathcal{B})}$ and $\|\nabla u_0\|_{L^\infty{(\der\mathcal{B})}}$ are bounded by a constant depending only on the apriori constants and, hence,
\begin{equation}\label{s2}\left|\int_{\der\mathcal{B}}\gamma_{\mathcal{P}^0}b\cdot r\right|\leq C
\end{equation}
For a similar reason, for point on $\der\mathcal{P}_0\setminus(\mathcal{B}\cup B(P,\frac{d_0}{C_0}))$ we can bound the $L^\infty$ norms of $u_0$ and $v_0$ for point $y$ and $z$ in $K$ and
 \begin{equation}\label{s3}
\left| \int_{\der\mathcal{P}_0\setminus(\mathcal{B}\cup B(P,\frac{d_0}{C_0}))}\left(\tilde{\mathcal{U}}\cdot n\right)(k-\gamma_{\mathcal{P}^0})\mathcal{M} \nabla u_0^i\cdot \nabla v_0^i \right|\leq C.
 \end{equation}
Now, finally, we consider 
\[\int_{\der\mathcal{P}_0\cap B(P,\frac{d_0}{C_0})}\left(\tilde{\mathcal{U}}\cdot n\right)(k-\gamma_{\mathcal{P}^0})\mathcal{M} \nabla u_0^i\cdot \nabla v_0^i\]
and notice that, if $y,z$ are at positive fixed distance from  $\der\mathcal{P}_0\cap B(P,\frac{d_0}{C_0})$, then we can again use \eqref{energyGreen} and \eqref{linir} and bound the $L^\infty$ norm of $\nabla u_0^i$ and $\nabla v_0^i$.
On the other hand, for points $y$ and $z$ close to $\der\mathcal{P}_0\cap B(P,\frac{d_0}{C_0})$ we can use \eqref{estGreen3} and the explicit formula for $\Gamma$ to finally get that

\begin{equation}
\label{s4}
\left|
\int_{\der\mathcal{P}_0\cap B(P,\frac{d_0}{C_0})}\left(\tilde{\mathcal{U}}\cdot n\right)(k-\gamma_{\mathcal{P}^0})\mathcal{M} \nabla u_0^i\cdot \nabla v_0^i
\right|\leq C\left(d_y d_z\right)^{-1/2}
\end{equation}
where $d_y=dist(y,\mathcal{P}^0)$ and $d_z=dist(z,\mathcal{P}^0)$ .

By putting together  \eqref{b62} and \eqref{s1}-\eqref{s4} we finally get, for every $y,z\in K$
\begin{equation}
\label{limit}
\left|S_0(y,z)
\right|\leq C\left(d_y d_z\right)^{-1/2}
\end{equation}

Let us now recall that  $S_0$ solves \eqref{diffequS} with respect to both variables, that is bounded by \eqref{12.1} if $y,z\in  D_0$ and is bounded by \eqref{limit} for $y,z\in K$.

Let us fix $z$ in $D_0$ and consider $S_0$ as a function of $y$ only. 

We take $Q_0=(0,L+d_0)$  and $\tilde{r}=\frac{d_0}{2}$, so that, by \eqref{12.1}
\begin{equation}\label{t1}\|S_0(\cdot,z)\|_{L^2(B_{\tilde{r}}(Q_0))}\leq C m_1.\end{equation}

Let us set $\tilde{K}=\{x\in K\,: d(x,\mathcal{P}^0)\geq \frac{d_0}{4C_0}\}$ and $U=\{x\in K_0\,: d(x,\mathcal{P}^0)\geq \frac{d_0}{2C_0}\}$.

By \eqref{limit} we have that there is a constant $C$ depending only on the a priori data, such that
\begin{equation}\label{t2}
\|S_0(\cdot,z)\|_{L^2(\tilde{K})}\leq C.
\end{equation}

We now apply Theorem 3.1 in \cite{CW} that gives an estimate on smallness propagation for solutions of elliptic differential equation with jumps. Notice that, in the present case, the right hand side of the equation is null and we take $\Omega=\tilde{K}$. There exists positive constants $\delta$ and $C$ depending only on the a priori data, such that, by \eqref{t1} and \eqref{t2} we have
\begin{equation}\label{t3}
\|S_0(\cdot,z)\|_{L^2(U)}\leq C m_1^\delta.
\end{equation}
In particular, since the disk centered at $\overline{Q}=P+\frac{5d_0}{8C_0}$ of radius $R_1=\frac{d_0}{8C_0}$ is contained in $U$, we have
\begin{equation}\label{t3bis}
\|S_0(\cdot,z)\|_{L^2\left(B_{R_1}(\overline{Q})\right)}\leq C m_1^\delta \quad \forall z\in D_0.
\end{equation}
Now let us notice that, for fixed $z$, $S_0$ is a harmonic function in 
$\{x\in K\,: dist(x,\mathcal{P}^0)\leq \frac{d_0}{C_0}\}$ hence, since \eqref{t3bis} holds, by the mean value property and H\"older inequality
\begin{equation}\label{t4}\left|S_0(y,z)\right| \leq  C m_1^\delta\end{equation}
for $y\in B_{R_1/2}(\overline{Q})$ and $z\in D_0$.

Let us now fix $y\in B_{R_1/2}(\overline{Q})$ and consider the function $S_0(y,\cdot)$. By the same procedure as before we have that
\begin{equation}\label{2-2s}
\|S_0(y,\cdot)\|_{L^2\left(B_{R_1}(\overline{Q})\right)}\leq C m_1^{\delta^2} \quad \forall y\in B_{R_1/2}(\overline{Q}).
\end{equation}

We now apply a classical three sphere inequality for harmonic functions (see, for example, \cite[Appendix E.2]{ADB}) to the function $S_0(\cdot,z)$ in the spheres
\[B_{\overline{R}_1}(\overline{Q})\subset B_{\overline{R}_2}(\overline{Q}) \subset B_{\overline{R}_3}(\overline{Q})\]
for 
\[\overline{R}_1=R_1/2,\quad  \overline{R}_2=\frac{d_0}{2C_0}-\frac{r}{2}\mbox{ and }
\overline{R}_3=\frac{d_0}{2C_0}-\frac{r}{4}.\]
We have that, for every $z\in B_{\overline{R}_1}(\overline{Q})$ and for 
\begin{equation}\label{expon}
\theta_r=\frac{\ln(\overline{R}_3/\overline{R}_3)}{\ln(\overline{R}_3/\overline{R}_1)}
\end{equation}
then
\begin{equation}\label{2-3s}
\|S_0(\cdot,z)\|_{L^2\left(B_{\overline{R}_2}(\overline{Q})\right)}\leq \|S_0(\cdot,z)\|^{\theta_r}_{L^2\left(B_{\overline{R}_1}(\overline{Q})\right)}\|S_0(\cdot,z)\|^{1-\theta_r}_{L^2\left(B_{\overline{R}_3}(\overline{Q})\right)},
\end{equation}
hence, by \eqref{limit} and \eqref{2-2s} we have
\begin{equation}\label{3-3s}
\|S_0(\cdot,z)\|_{L^2\left(B_{\overline{R}_2}(\overline{Q})\right)}\leq C\left(\frac{1}{r}\right)^{\frac{1-\theta_r}{2}}m_2^{\theta_r} \quad \forall z\in B_{\overline{R}_1(\overline{Q})}
\end{equation}
where (see \eqref{2-2s})
\[m_2=Cm_1^{\delta^2}.\]
Recalling that $S_0(\cdot,z)$ is a harmonic function, by the mean value properties and H\"older inequality, we have
\begin{equation}\label{4-3s}
|S_0(y_r,z)|\leq C\left(\frac{1}{r}\right)^{1+\frac{1-\theta_r}{2}}m_2^{\theta_r} \quad \forall z\in B_{\overline{R}_1(\overline{Q})}.
\end{equation}
We now consider $S_0(y_r,\cdot)$ in the same disks as before and we finally get
\begin{equation}\label{4-4s}
|S_0(y_r,y_r)|\leq C\left(\frac{1}{r}\right)^{1+\theta_r + (1-\theta_r)(\frac{\theta_r}{2}+1)}m_2^{\theta_r^2} \end{equation}
From (\ref{expon}) it is straightforward to see that 
\[
\left(\frac{1}{r}\right)^{\theta_r}=1+o(1),\,\text{ as }r\rightarrow 0
\]
which implies the following 
\begin{equation}\label{4-5s}
|S_0(y_r,y_r)|\leq C\frac{1}{r^2}m_2^{\theta_r^2}. 
\end{equation}

By using estimate (\ref{estGreen3}) and proceeding similarly as in \cite[(3.21)]{BF}  we have for $r\leq \frac{d_0}{8C_0}$
\begin{equation}\label{5-1s}
	\left|S_0(y_r,y_r)\right|\geq  \frac{C\tilde{\mathcal{U}}(P)\cdot n(P)}{r}-C\left|\ln r)\right|\end{equation}
and, by comparing  \eqref{4-5s} with \eqref{5-1s}  we get
\begin{equation}\label{5-2s}
\left|\mathcal{U}(P)\cdot n(P)\right|\leq C\left(r\left|\ln r\right|+\frac{1}{r}m_2^{\theta_r^2}\right).
\end{equation}
If $m_2\leq \exp (-(48)^4)$ i.e. $m_1\leq \left(\frac{ \exp( -(48)^4) }{C}\right)^{1/\delta^2}$ then we can pick up $$r=96 \overline{R}_1|\ln (m_2)|^{-1/4}=\frac{6d_0}{C_0}|\ln (m_2)|^{-1/4}$$ in \eqref{5-2s} and after some straightforward estimation  we end up with the following bound  
\[
\left |\mathcal{U}(P)\cdot n(P)\right |\leq C|\ln (m_2)|^{-1/5}
\]
which, recalling the definition of $m_2$, can be written as
\begin{equation}\label{omegazero}|\tilde{\mathcal{U}}(P)\cdot n(P)|\leq  \omega_0(m_1),\end{equation}
where $\omega_0(t)$ is an increasing concave function such that $\lim_{t\to 0^+}\omega_0(t)=0$. 

Notice that, with a similar procedure, this estimate can be obtained for each point in a neighborhood of $P$ on the side $P^0_jP^0_{j+1}$.
Since $\tilde{\mathcal{U}}$ is affine on $P^0_jP^0_{j+1}$, we get that the estimate holds at the endpoint of the segment as well, hence
\[|\tilde{\mathcal{U}}(P^0_j)\cdot n_j)|\leq \omega_0(m_1),\]
and 
\[|\tilde{\mathcal{U}}(P^0_{j+1})\cdot n_j|\leq \omega_0(m_1),\]
where $n_j$ denotes the outer normal direction to $P^0_jP^0_{j+1}$ and $\omega_0$ is the  function  in \eqref{omegazero} multiplied by a constant depending on the a priori parameters.

 In particular, by \eqref{i.31}, for each $P^0_j$ we have
\begin{equation}\label{b101}\left|\frac{P^0_j-P^1_j}{|V|}\cdot n\right|\leq \omega_0(m_1)\end{equation}
for each $n$ normal to sides through $P^0_j$.

Let $j_0$ be such that $|P^0_{j_0}-P^1_{j_0}|=\max_{j=1,\ldots,M}|P^0_{j}-P^1_{j}|$.  By construction $P^0_{j_0}$ is an endpoint of $\mathcal{P}^0$ (and not an intersection with $\Sigma$) and 
$\frac{|P^0_{j_0}-P^1_{j_0}|}{|V|}\geq \frac{1}{M}$.
Moreover, since there are two linearly independent unit directions $n$ for which \eqref{b101} holds for $j=j_0$, then it holds for every unit direction, and, by choosing $\tilde{n}$ parallel to $P^0_{j_0}-P^1_{j_0}$ we get
\[\frac{1}{M}\leq \frac{|P^0_{j_0}-P^1_{j_0}|}{|V|}=\left|\frac{P^0_j-P^1_j}{|V|}\cdot \tilde{n}\right|\leq \omega_0(m_1)\]
from which
\begin{equation}\label{m1s}
m_1\geq \omega_0^{-1}(\frac{1}{M}).
\end{equation}
By definition of the operator norm of $H$, there exist $f_0$ and $g_0$ in $H^{1/2}(\der\om)$ such that
\[|H(f_0,g_0)|\geq \frac{m_1}{2}\|f_0\|_{H^{1/2}(\der\om)}\|g_0|_{H^{1/2}(\der\om)}\]
and \eqref{tesibasso} is true for $m_0=\frac{\omega_0^{-1}(1/N)}{2}$. Finally observe that if $m_1> \left(\frac{ \exp( -(48)^4) }{C}\right)^{1/\delta^2}$ then the statement is true for $m_0=\frac{1}{2}\left(\frac{ \exp( -(48)^4) }{C}\right)^{1/\delta^2}$ which concludes the proof.
\end{proof}
\begin{rem}
The same proof works in the case of multiple interfaces as long as we control the distance between interfaces (see assumption \eqref{distanzastrati}) and only consider polygons whose vertices are far from the interfaces (see assumption \eqref{distanzainterfaccia}). In that case the proof is only technically more involved to write down in the integration by parts that lead to formula \eqref{b62}. As far as the unique continuation estimate from \cite{CW} that we used to get estimate \eqref{t3}, they rely on a three ball inequality that holds even if the ball intersect the interfaces (see \cite[Theorem 4.1]{CW}).
\end{rem}

\subsection{Lipschitz stability estimate}
In this section we conclude the proof of Theorem \ref{mainteo}. 

Let us first assume that $\|\Lambda_{\mathcal{P}^0}-\Lambda_{\mathcal{P}^1}\|_*\leq \varepsilon_0$ and let $f_0$ and $g_0$ the functions the satisfy \eqref{tesibasso} in Proposition \ref{p3.3}.
By \eqref{tesibasso} and by \eqref{contder} we have

\begin{eqnarray*}\label{fine}
	 \left|<\left(\Lambda_{\mathcal{P}^0}-\Lambda_{\mathcal{P}^1}\right)(f_0),g_0>\right|\!\!&=\!\!&\left|F(1,f_0,g_0)-F(0,f_0,g_0)\right|=\left|\int_0^1F'(t,f_0,g_0)dt\right|\\
\!\!	&\geq\!\!& |F'(0,f_0,g_0)| \!-\!\int_0^1\!\left| F'(t,f_0,g_0)-F'(0,f_0,g_0)\right|dt\\
	&\geq&\left(m_0-C|V|^\beta\right)|V|\|f_0\|_{H^{1/2}(\der\om)}\|g_0\|_{H^{1/2}(\der\om)}.
\end{eqnarray*}
Since
$|V|\leq M\max_jdist(P^0_j,P^1_j)$, by \eqref{stimarozza} it follows that there exists $\ep_1\in(0,\ep_0)$ depending only on a priori constant such that, if
\[\|\Lambda_{\mathcal{P}^0}-\Lambda_{\mathcal{P}^1}\|_*\leq\ep_1,\] then
\[\left(m_0-C|V|^\beta\right)\geq m_0/2\]
and
\begin{equation}\label{vertici}|V|\leq \frac{2}{m_1}\|\Lambda_{\mathcal{P}^0}-\Lambda_{\mathcal{P}^1}\|_*.\end{equation}
Finally, since 
\[d_{H}\left( \partial \mathcal{P}^{0},\partial \mathcal{P}^{1}\right)\leq C| V|\]
 the claim follows if $\|\Lambda_{\mathcal{P}^0}-\Lambda_{\mathcal{P}^1}\|_*\leq\ep_1$ (where $\ep_1$ depend only on the a priori data).

If, now, $\|\Lambda_{\mathcal{P}^0}-\Lambda_{\mathcal{P}^1}\|_*\geq\varepsilon_1$ (and, hence, also if $\|\Lambda_{\mathcal{P}^0}-\Lambda_{\mathcal{P}^1}\|_*\geq\varepsilon_0$), since the
following trivial inequality  holds 
\[
d_{H}\left( \partial \mathcal{P}^{0},\partial \mathcal{P}^{1}\right) \leq 2L,
\]
we easily derive%
\begin{equation}
d_{H}\left( \partial \mathcal{P}^{0},\partial \mathcal{P}^{1}\right) \leq
2L\leq 2L\frac{\| \Lambda_{\mathcal{P}^0}-\Lambda_{\mathcal{P}^1}\| _{\ast }}{%
\varepsilon _{1}}\leq C \| \Lambda_{\mathcal{P}^0}-\Lambda_{\mathcal{P}^1}\| _{\ast } \label{stab2}.
\end{equation}
\qed
\section*{Acknowledgements}
Elisa Francini and Sergio Vessella were partially supported by Research Project 201758MTR2 of the Italian Ministry of Education, University and Research (MIUR) Prin 2017 “Direct and inverse problems for partial differential equations: theoretical aspects and applications”.

\end{document}